\documentclass[leqno,sumlimits,intlimits,namelimits]{amsart}

\usepackage{bm}

\usepackage{amssymb,latexsym}
\usepackage[numeric,initials,bibtex-style,nobysame]{amsrefs}
\usepackage{amsthm}

\usepackage{a4wide}
\usepackage[american]{babel}

\usepackage{eucal}
\usepackage{mathrsfs}

\usepackage[dvipsnames]{xcolor}

\usepackage{mathtools}
\mathtoolsset{showonlyrefs=true}

\def\pg{\mathhexbox278}

\newcommand{\D}{\ensuremath{\mathcal{D}}}

\newcommand{\loc}{\ensuremath{\text{loc}}}

\newcommand{\mb}[1]{\ensuremath{\mathbb{#1}}}
\newcommand{\N}{\mb{N}}

\newcommand{\R}{\mb{R}}

\newcommand{\Z}{\mb{Z}}
\newcommand{\T}{\mb{T}}

\newcommand{\sgn}{\mathop{\mathrm{sgn}}}

\renewcommand{\d}{\ensuremath{\partial}}
\newcommand{\diff}[1]{\frac{d}{d#1}}

\renewcommand{\div}{\mathop{\mathrm{div}}}

\newfont{\bl}{msbm10 scaled \magstep2}

\newtheorem{theorem}{Theorem}[section]
\newtheorem{lemma}[theorem]{Lemma}
\newtheorem{proposition}[theorem]{Proposition}
\newtheorem{definition}[theorem]{Definition}

\theoremstyle{definition}
\newtheorem{remark}[theorem]{Remark}
\newtheorem{example}[theorem]{Example}

\newcommand{\beq}{\begin{equation}}
\newcommand{\eeq}{\end{equation}}

\newcommand{\col}{\colon}

\newcommand{\FT}[1]{\widehat{#1}}

\newcommand{\dis}[2]{\langle #1 , #2 \rangle}
\newcommand{\notmid}{\mid\kern-0.5em\not\kern0.5em}

\newcommand{\norm}[2]{{\left\| #1 \right\|}_{#2}}

\newcommand{\de}{\delta}
\newcommand{\eps}{\varepsilon}

\newcommand{\vphi}{\varphi}

\newcommand{\la}{\lambda}

\newcommand{\om}{\omega}

\newcommand{\ovl}[1]{\overline{#1}}

\title{Solution concepts, well-posedness, and wave breaking for the Fornberg-Whitham equation}
\author{G\"unther H\"ormann}

\address{G\"unther H\"ormann: Fakult\"at f\"ur Mathematik, Universit\"at Wien, Austria}
\email{guenther.hoermann@univie.ac.at}

\date{\today}

\begin{document}
\pagestyle{plain}
\maketitle	

\begin{abstract} We discuss concepts and review results about the Cauchy problem for the Fornberg-Whitham equation, which has also been called Burgers-Poisson equation in the literature. Our focus is on a comparison of various strong and weak solution concepts as well as on blow-up of strong solutions in the form of wave breaking. Along the way we add aspects regarding semiboundedness at blow-up, from semigroups of nonlinear operators to the Cauchy problem, and about continuous traveling waves as weak solutions. 
 \end{abstract}

\tableofcontents

\section{Introduction and basic set-up}

The intention of this review-type article is to put some of the key mathematical notions and solution results regarding the \emph{Fornberg-Whitham equation} in a perspective with respect to each other and we will thereby also strive to connect two so far largely parallel threads of research, because the same equation has also been studied under the name of \emph{Burgers-Poisson equation}. Neither do we attempt here to elaborate on the history and physics behind this model equation nor can we come anywhere near a complete overview of  mathematical results from the more than 50 years of its analysis. Moreover, our attention was restricted to results from work published at the time of writing and no systematic search of preprints was undertaken.

We discuss here the Fornberg-Whitham equation as it was introduced by Whitham in \cite[Equation (67)]{Whitham:67} as the integro-differential equation at the center of a shallow water wave model that is comparably simple and yet showed indications of wave breaking (see also \cite{Seliger68}). It featured later in Whitham's book \cite{WhithamBook} in a section dedicated to breaking and peaking of waves and   a first systematic numerical study was published by Fornberg and Whitham in \cite[Section 6]{FB78}. 

Let us describe the formal set-up of the Cauchy problem. The wave height is described by a function of one-dimensional space and time $u \col \R \times [0,\infty[ \to \R$, $(x,t) \mapsto u(x,t)$. We will occasionally write $u(t)$ to denote the function $x \mapsto u(x,t)$. Upon rescaling (cf.\ Remark \ref{rem:varFW} below) we may write the equation without explicitly occurring additional model parameters in the form
\begin{equation}
  u_t + u u_x + K*u_x = 0,
  \label{eqn:F-W}
\end{equation}
where the convolution is in the $x$ variable only and $t > 0$. The convolution kernel is $K(x) = \frac{e^{-|x|}}{2}$ and satisfies 
\beq\label{eqn:Kprop}
   K - K'' = \delta, 
\eeq   
which means that $K$ is a fundamental solution of the operator $1 - \d_x^2$. 
In fact, we will occasionally have to interpret Equation \eqref{eqn:F-W} in various weak forms---with distributional, entropy, or mild semigroup solution concepts---which stem from rewriting the left-hand side either as in
\begin{equation}
  \partial_t u + \partial_x \left( \frac{u^2}{2} + K*u \right) = 0
  \label{eqn:F-Wrew}
\end{equation}
or also in the form
\begin{equation}
  \partial_t u + \partial_x \left( \frac{u^2}{2}\right) + K' * u  = 0.
  \label{eqn:F-Wrew2}
\end{equation}

Based on the property \eqref{eqn:Kprop}, Equation \eqref{eqn:F-W} emerged in \cite{FS2004} instead from a system of equations, which can be approached here in reverse direction upon rewriting \eqref{eqn:F-W} as $u_t + u u_x = - K*u_x$ and putting $v := - K * u$. Noting that $v_x = - K * u_x$ and $(1 - \d_x^2) v = - u$ we then obtain the following system of nonlinear partial differential equations
\begin{align*}
     u_t + u u_x &= v_x,\\
     v_{xx} &= v + u.
\end{align*}
It was the starting point of the model in \cite{FS2004} and called \emph{Burgers-Poisson system}, while the analog of Equation  \eqref{eqn:F-W} derived from it got named \emph{Burgers-Poisson equation}. This name was also used in the key publication about global weak solutions in \cite{GNg:16}. In the context of the current review article  we prefer to stay with the notion of Fornberg-Whitham equation referring to \eqref{eqn:F-W}.

We will usually suppose an initial wave profile $u_0 \colon \R \to \R$ to be given and require in addition
\begin{equation}
  u|_{t=0} = u_0.
  \label{ic:F-W}
\end{equation}

\begin{remark}[Rescaled and periodic variants of the Fornberg-Whitham equation]\label{rem:varFW} 
(i)  Note that we followed here in \eqref{eqn:F-W} the sign convention for the convolution term as in \cite[Equation (4)]{FB78}\footnote{It does not agree with all signs in Equation (29) of \cite{FB78}, since that equation contains a sign error with the linear term involving $u_x$.} and \cite[Section 13.14]{WhithamBook}, but have applied a rescaling of the solution values in order to get rid of any additional constant factor in the nonlinear term. Replacing $u(x,t)$ by $-u(x,-t)$ transforms solutions of either sign variant of the equation into solutions for the other convention. Moreover, if $u$ solves \eqref{eqn:F-W} and $\la > 0$ is a constant, then $v := u/\la$ is a solution to 
$$
   v_t + \la v v_x + K * v_x = 0, 
$$
which shows why we could bring the original model equation from \cite{Whitham:67} into the form  \eqref{eqn:F-W}. Such scalings and sign conventions have to be taken into account when comparing results about wave breaking that typically involve also quantitative aspects of the initial wave profile.

\noindent (ii) Formally applying $1 - \d_x^2$ to \eqref{eqn:F-W} produces the third order partial differential equation 
\beq\label{eqn:order3}
     u_t - u_{txx}  - 3 u_x u_{xx} - u u_{xxx} +  u u_x + u_x  = 0.
\eeq
Instead we will stay with the non-local integro-differential equation \eqref{eqn:F-W} or \eqref{eqn:F-Wrew}, because it corresponds to the original model and is also more suitable for the various solution concepts to be discussed.

\noindent (iii) To study spatially periodic waves we change the $x$-domain to the one-dimensional torus group $\T = \R / \Z$ and may identify functions on $\T$ with $1$-periodic functions on $\R$. This also requires an adaptation of the convolution kernel $K$ (cf.\  \cite[Section 3]{GH:18}), which is then given as the $1$-periodic function on $\R$ with $K(x) = (e^x + e^{1-x})/(2 (e-1)) = \frac{\sqrt{e}}{e-1}
 \cosh(x - \frac{1}{2}) $ for $0 \leq x < 1$. Note that $K$ is continuous but not $C^1$ and the derivative $K'$ is not continuous but in $L^\infty$.   
\end{remark}

To simplify the presentation in the context of this review, we will give detailed formulations only for the Fornberg-Whitham equation in the form \eqref{eqn:F-W} and without periodicity assumptions. However, we will occasionally add remarks on the periodic case.

Before discussing  in the following section the main solution concepts that have been employed for the Cauchy problem consisting of \eqref{eqn:F-W} and \eqref{ic:F-W}, let us remark that there are not many conserved quantities for solutions $u$ (of sufficient regularity and with suitable integrability properties). The most obvious one is
\beq\label{consint}
   \forall t \geq 0\col \quad \int_\R u(t,x)\, dx = \int_\R u_0(x)\, dx,
\eeq 
since integrating \eqref{eqn:F-W} with respect to $x$ gives
\begin{multline*}
    \diff{t} \int u(x,t)\, dx = \int \d_t u(x,t) \, dx = - \int \Big(u(x,t) \d_x u(x,t) + (K * \d_x u(.,t))(x) \Big) \, dx\\
    = -  \frac{1}{2} \int \d_x (u(x,t)^2) \, dx - \int \d_x (K * u(.,t))(x) \, dx = - \frac{1}{2} \cdot 0 - 0 = 0.
\end{multline*}
A second conserved quantity is the spatial (real) $L^2$ norm and stems from the skew-symmetry\footnote{The symmetry of $K$ and Fubini's theorem imply $\dis{K * v}{w} = \dis{v}{K * w}$ and an additional differentiation gives skew-symmetry.} of the operator $v \mapsto K' * v = (K * v)'$ on $L^2(\R)$: Multiplying \eqref{eqn:F-W} by $u$, integrating with respect to $x$, and writing $u u_t = \diff{t} (u^2/2)$, $u^2 u_x = \diff{x}(u^3/3)$ we obtain
$$
    0 = \frac{1}{2} \diff{t} \int u(x,t)^2 \, dx + \frac{1}{3} \int \diff{x}(u(x,t)^3)\, dx + \dis{K' * u}{u} =  
    \frac{1}{2} \diff{t} \norm{u(t)}{L^2} +  \frac{1}{3} 0 + 0 =  \frac{1}{2} \diff{t} \norm{u(t)}{L^2},
$$
and thus (see also \cite[Lemma 1]{GLC2018})
\beq\label{consL2}
      \forall t \geq 0\col \quad \int u(x,t)^2 \, dx =  \int u_0(x)^2 \, dx.
\eeq
However,  thanks to the analysis in \cite{Ivanov2005}, the Fornberg-Whitham equation is known to belong to those equations among a class of 3rd order nonlinear dispersive wave equations that are definitely not completely integrable. Therefore, the key methods from geometric theories of infinite-dimensional dynamical systems that are available, e.g., for the Camassa-Holm equation, are not applicable in case of the Fornberg-Whitham equation. 

The structure of this article is as follows: Section 2 is devoted to a discussion and comparison of various strong and weak solution concepts for the Cauchy problem consisting of \eqref{eqn:F-W} and \eqref{ic:F-W}. In Section 3 we summarize the key well-posedness results for strong solutions and on blow-up in finite time in the form of wave breaking, where we also add one aspect of semi-boundedness at blow-up time. Section 4 discusses key results on weak entropy solutions and adds a brief investigation of mild solutions with their relations to the former. The final subsection then focusses on continuous traveling waves in relation to the weak or weak entropy solution concept.

\section{Solution concepts for the Cauchy problem}

\subsection{Strong solutions} 

In pure classical terms, the minimum requirements for a particular function $u \col \R \times [0,\infty[ \to \R$ to count as a global solution of the Cauchy problem consisting of \eqref{eqn:F-W} and \eqref{ic:F-W}, would be like the following:  $u$ possesses first-order partial derivatives in $\R \times\, ]0,\infty[$,  the convolution $K * u_x(.,t)$ is defined on $\R$ for every $t > 0$, Equation \eqref{eqn:F-W} holds pointwise for every $(x,t) \in \R \times\, ]0,\infty[$, and $u(x,0) = u_0(x)$ for all $x \in \R$.  We could instead also consider classical solution for a finite time interval $[0,T[$ instead of $[0,\infty[$ and the adaptations in the conditions described above are obvious. 

\begin{remark} In the context of  partial differential conservation laws the term \emph{classical solution} is also used, e.g., by Dafermos (cf.\ \cite[Section 4.1]{Dafermos2016}), for locally Lipschitz continuous functions $u$ that satisfy the differential equation almost everywhere on the $(x,t)$-domain. We could mimic this here with bounded and locally Lipschitz continuous functions, since $K' \in L^1(\R)$ so that the convolution $K' * u$ is defined. This lies somewhat between typical weak solution concepts and what we will call strong solution below.
\end{remark}

A somewhat restrictive, but more systematic and modern approach is to first identify some topological multiplicative algebra $X$ of functions on the real line that is invariant under differentiation $\d_x$ and such that the operator of convolution with $K$ acts continuously $X \to X$. A standard example is $X = H^\infty(\R)$. 
Assuming $u_0 \in X$, one then searches for a solution on $[0,T[$ of \eqref{eqn:F-W} and \eqref{ic:F-W} in the sense\footnote{differentiability of a map $v \col ]0,T[ \to X$ at $t \in\, ]0,T[$ simply meaning that the difference quotient $(v(t_1) - v(t))/(t_1-t)$ converges to a limit $v'(t)$ in $X$ as $t_1 \to t$; and $v \in C^1(]0,T[,X)$ then requires that $v$ is differentiable at every $t \in\, ]0,T[$ and $t \mapsto v'(t)$ is continuous $]0,T[\, \to X$} that $u \in C^1(]0,T[, X) \cap C([0,T[, X)$ should satisfy $u(0) = u_0$ and \eqref{eqn:F-W} holds as an equation in $X$ for $0 <  t < T$, i.e., 
\beq\label{strongeqn}
     \forall t \in \R, 0 < t < T\col \quad u'(t) + u(t) \d_x u(t) + K * \d_x u(t) = 0.
\eeq
In cases where $T$ may be taken arbitrarily large we speak of a solution \emph{global in time}. 

\begin{remark} For any $u$ satisfying Equation \eqref{strongeqn} in the above sense we have $u'(t) = - u(t) u_x(t) - K * u_x(t)$ for all $t > 0$, where the right-hand side belongs to $C([0,T[,X)$. Therefore, $u' = \d_t u$ can be continuously extended to $t = 0$ and we may thus specify  $u \in C^1([0,T[, X) \cap C([0,T[, X)$ from the outset.
\end{remark}

The required invariance of $X$ under differentiation makes it hard to obtain $X$ itself as a Banach algebra, but an alternative is to resort to a scale of Banach spaces $X^s$ ($s \in [0,\infty[$) with $X^{s_2} \hookrightarrow X^{s_1}$, if $s_1 \leq s_2$ and differentiation being continuous $X^{s+1} \to X^s$. The standard examples are Sobolev-type spaces, in particular, $X^s = H^s(\R)$ with $X^0 = L^2(\R)$. In the latter case, we also know that we obtain a Banach algebra, if $s > 1/2$ (and we adapt the Sobolev norm by an appropriate constant factor; cf.\ \cite[Theorem 4.39]{AF:03}). Observe that moreover,  $v \mapsto K*v$ is a continuous operator on $H^s(\R)$ for every $s \geq 0$, since $K \in L^1(\R)$ (hence $\FT{K * v} = \FT{K} \cdot \FT{v}$ with $\FT{K}$ continuous and bounded).  
Therefore, the Sobolev spaces $H^s(\R)$ (with $s  >  1/2$) provide an example of the following set-up.

Suppose $s_0 \geq 0$ and $X^s$ ($s > s_0$) is a scale of Banach algebras of function spaces on $\R$, $X^{s_2} \hookrightarrow X^{s_1}$ ($s_0 < s_1 \leq s_2$), the product in $X^s$ being pointwise multiplication of functions, and such that convolution by $K$ acts continuously on every space and differentiation is continuous $X^{s+1} \to X^s$. Let $s > s_0$ and $0 < T \leq \infty$. 
A \emph{strong solution} on the time interval $[0,T[$ of the Cauchy problem \eqref{eqn:F-W} and \eqref{ic:F-W} with initial value $u_0 \in X^{s + 1}$ is given by an element $u \in C^1([0,T[,X^s) \cap C([0,T[, X^{s + 1})$ such that $u(0) = u_0$ and \eqref{strongeqn} holds as an equation in $X^{s}$ for $0 <  t < T$.

A typical notion of \emph{well-posedness of the Cauchy problem} \eqref{eqn:F-W} and \eqref{ic:F-W} will require that, given any $u_0 \in X^{s + 1}$, there is some $0 < T \leq \infty$ such that a unique solution $u$ to \eqref{strongeqn} with $u(0) = u_0$ exists in $C^1([0,T[,X^s) \cap C([0,T[, X^{s + 1})$ and that the solution map $u_0 \mapsto u$ is continuous, e.g., for every $T_1 \in\, ]0,T[$ as a map between the Banach spaces $X^{s +1} \to C([0,T_1], X^{s + 1})$, where the norm on $C([0,T_1], X^{s + 1})$  is $\sup_{0 \leq t \leq T_1} \norm{u(t)}{X^{s+1}}$ {\small (the supremum exists, because $[0,T_1]$ is compact, and this was the reason for taking $T_1 < T$)}. A reasonable variant of the notion may speak of \emph{well-posedness on the closed finite time interval} $[0,T_1]$, if the solution exists and is unique in $C^1([0,T_1],X^s) \cap C([0,T_1], X^{s + 1})$ with continuity of $u_0 \mapsto u$ as above. 

Proofs of well-posedness typically establish a so-called \emph{a priori estimate} of $\sup_{0 \leq t \leq T_1} \norm{u(t)}{X^{s+1}} $ in terms of some concrete bounded function of the \emph{life span} $T_1$, the regularity $s$, and $\norm{u_0}{X^{s+1}}$.
In case of well-posedness, the \emph{maximal life span} $T$ associated with a given regularity $s > s_0$ and an initial value $u_0 \in X^{s+1}$ is the supremum of all $T_1 > 0$ such that a (unique) solution exists in $C^1([0,T_1],X^s) \cap C([0,T_1], X^{s + 1})$ with $u(0) = u_0$. To have a unique strong solution global in time thus means that we have the maximal life span $T = \infty$. On the other hand, a situation with finite maximal life span, i.e.,  $T < \infty$, does lead to \emph{blow-up} of the solution in finite time and is also the starting point for discussions of the question of \emph{wave breaking}.  

\medskip

\noindent\textbf{Blow-up of strong solutions and wave breaking:}
Suppose now that for given $s > s_0$ and initial wave profile $u_0 \in X^{s+1}$ we have maximal life span $T < \infty$. Since $u \in C^1([0,T[,X^s)$ and $u \in C([0,T[,X^{s+1})$ at least one of the following two situations has to arise: (a) There is no continuous extension of $t \mapsto u(t)$, $[0,T[\, \to X^{s+1}$ at $t=T$; (b) $t \mapsto u(t)$, $[0,T[\, \to X^s$ cannot be extended as a continuously differentiabe map up to $t = T$. We claim that (a) must hold, i.e.,
\beq\label{nonext}
    \text{the map $t \mapsto u(t)$, $[0,T[\, \to X^{s+1}$, cannot possess a continuous extension to $t = T$},
\eeq 
for we could otherwise extend $u$ to a solution with a life span larger than $T$: Indeed, if \eqref{nonext} is false,
 then $v_0 := \lim_{t \to T} u(t) \in X^{s+1}$ can serve as an initial value with some life time $T_0 > 0$ for a unique solution $v \in C^1([0,T_0], X^s) \cap C([0,T_0],X^{s+1})$. We patch the two solutions $u$ and $v$ into one function $w \in C([0,T+T_0],X^{s+1})$, i.e., $w(t) := u(t)$ ($0 \leq t \leq T$), $w(t) := v(t)$ ($T < t \leq T+T_0$), which obviously solves \eqref{strongeqn} for $t \neq T$ and satisfies $w \in C^1([0,T[\, \cup\, ]T, T+ T_0], X^s)$. We have to show that $w$ is $C^1$ at $t = T$ and also solves \eqref{strongeqn} there. We clearly have $w'(T+) := \lim_{t \downarrow T} w'(t) = v'(0)$. The limit from the left, $w'(T-) := \lim_{t \uparrow T} w'(t)$ also exists, since again by \eqref{strongeqn} we may represent $w'(t)$ in terms of $w(t)$ and $w_x(t)$ and both are continuous at $t = T$ with values in $X^s$ by the negation of \eqref{nonext}. This yields $w'(T-) = - v(0) \d_x v(0) - K * \d_x v(0) = \d_t v(0) = v'(0)$. Thus, $w$ is $C^1$ also at $t = T$ and the validity of \eqref{strongeqn} on all of $ [0,T+T_0]$ follows from continuity of all terms in it. 

The consequence of the finiteness of the life span $T$ expressed in \eqref{nonext}  means that $u(t)$ does not converge in $X^{s+1}$ as $t \to T$. In this generality we do not see how to assess whether $\norm{u(t)}{X^{s+1}}$ stays bounded or we have \emph{blow-up} of the solution at $t = T$ in the sense that
$$
    \limsup_{t \uparrow T} \norm{u(t)}{X^{s+1}} = \infty.
$$

In the most prominent case with the scale of Sobolev spaces $H^s(\R)$ ($s > 1/2$), we have that $H^{s+1}(\R)$ is continuously embedded in the space of bounded $C^1$ functions with bounded derivative (\cite[Chapter IV, \pg3, Theorem 1]{DLV2}). Thus, we obtain that for any strong solution $u \in C^1([0,T[,H^s(\R)) \cap C([0,T[,H^{s+1}(\R))$ the norms $\norm{u(t)}{L^\infty}$ and  $\norm{\d_x u(t)}{L^\infty}$ are finite for every $t \in [0,T[$. We say that \emph{wave breaking} occurs for $u$ at time $T > 0$  (cf.\ \cite[Definition 6.1]{Constantin2011}), if the wave itself remains bounded while its slope becomes unbounded at $t = T$, i.e., 
\beq\label{wavebreaking}
  \sup_{t \in [0,T[} \norm{u(t)}{L^\infty} < \infty \quad \text{and} \quad 
  \limsup_{t \uparrow T} \norm{\d_x u(t)}{L^\infty} = \infty.
\eeq
An \emph{analysis of wave breaking} should ideally address at least the following two issues:\\
(a) Whether a finite maximal life span $T < \infty$ for a strong solution $u$ necessarily implies  wave breaking for this solution at time $T$.\\ 
(b) Identification of a certain class of initial wave profiles $u_0$ such that the maximal life span of the corresponding strong solution is indeed finite. Hence wave breaking does definitely occur for the strong solutions with these initial values.

\subsection{Weak(er) solution concepts}

\subsubsection{Weak solutions} 

We recall how to proceed in the well-known standard policy to obtain an interpretation of the Cauchy problem \eqref{eqn:F-W} and \eqref{ic:F-W} in a weak or distributional sense: Suppose $u$ is a classical solution, write out \eqref{eqn:F-W} in the form \eqref{eqn:F-Wrew} or \eqref{eqn:F-Wrew2}, multiply the equation by an arbitrary test function $\phi$ from the space $\D(\R^2) := C^\infty_\text{c}(\R^2)$ of smooth functions with compact support, and integrate with respect to $x$ over all of $\R$ and with respect to $t$ over the half-line $[0,\infty[$; integration by parts and observing $u(x,0) = u_0(x)$ then yields an integral identity, which reads
$$
      \int_0^\infty \int_{-\infty}^\infty \Big(
    - u(x,t) \partial_t \phi(x,t) - \frac{u^2(x,t)}{2} \partial_x \phi(x,t) +  \big(K' \ast u(.,t)\big)(x) \phi(x,t) \Big) \, dx \, dt  = 
  \int_{-\infty}^\infty u_0(x) \phi(x,0)\, dx.
$$ 
Since $K' \in L^1(\R)$, it does make sense also for measurable functions $u$ and $u_0$ such that $u$ is bounded on $\R \times [0,T]$ for every $T > 0$ and $u_0$ is (locally) bounded. 
\begin{definition}\label{dfn:weaksol}
A measurable function $u \col \R \times [0,\infty[ \to \R$ that is bounded on $\R \times [0,T]$ for every $T > 0$ is called a \emph{weak solution} of the Cauchy problem \eqref{eqn:F-W}, \eqref{ic:F-W} with initial value $u_0 \in L^\infty(\R)$, if 
\begin{multline}\label{weaksol}
  \int_0^\infty \int_{-\infty}^\infty \Big( 
     u(x,t) \partial_t \phi(x,t) + \frac{u^2(x,t)}{2} \partial_x \phi(x,t) -  
    \big(K' \ast u(.,t)\big)(x) \phi(x,t) \Big) \, dx \, dt \\  
  + \int_{-\infty}^\infty u_0(x) \phi(x,0)\, dx = 0
\end{multline}
holds for every test function $\phi \in \D(\R^2)$.
\end{definition}
In the situation of the above definition, let us extend $u$ to a measurable function on all of $\R^2$ by setting $u(x,t) := 0$ for negative $t$. Upon writing $K' * u = \d_x(K * u)$ this leads to the distributional identity 
$$
   \div\nolimits_{(x,t)} (A_1,A_2) = u_0 \otimes \de \quad \text{ on $\R^2$},
$$ 
where $A_1 := u^2/2 + K * u$ and $A_2 := u$.  We  may thus deduce from \cite[Lemma 1.3.3]{Dafermos2016} (similarly as in the discussion in \cite[Section 4.3]{Dafermos2016}) the following, upon possibly modifying $u$ on a set of measure zero: For any relatively compact open subset $B$ of $\R$, the map $t \mapsto u(.,t)|_{B}$ is weak* continuous from $[0,\infty[$ into $L^\infty(B)$. This implies that $u$ may be considered a continuous map 
$$
   [0,\infty[ \to \D'(\R), t \mapsto u(t), 
$$   
where $u(t)$ is defined, for any $t \geq 0$, by its action on test functions $\vphi \in \D(\R)$ as 
$$
   \dis{u(t)}{\vphi} := \int u(x,t) \vphi(x)\, dx. 
$$   
In particular, we obtain that $\lim_{t \to 0} \int_B (u(x,t) - u_0(x)) \phi(x)\, dx = 0$ for any test function $\vphi \in \D(\R)$, i.e.,
$$
      u(t) \to u_0 \text{ in  $\D'(\R)$ as $t\to 0$,}\quad \text{or simply,  $u(0) = u_0$ holds in $\D'(\R)$}, 
$$
which also explains in what sense the initial value is attained for a weak solution according to Definition \ref{dfn:weaksol}. As the following remark illustrates, we cannot hope for a considerably stronger notion of continuity.

\begin{remark}\label{wnots}
In general, even with $B$ compact, the weak* continuity of a map $v \col [0,\infty[ \to L^\infty(B)$ does not imply strong continuity of $v$ as a map into some $L^p(B)$ although $L^\infty(B) \subseteq L^p(B)$. {\small (An example with $B = [0,1]$ is provided by $v \col [0,\infty[ \to L^\infty([0,1])$, where $v(t)(x) := \exp(ix/t)$, if $t > 0$, and $v(0) := 0$: Strong discontinuity of $v$ at $t = 0$ is obvious from $\norm{v(t) - v(0)}{L^p} = 1$ for all $t > 0$;  for any $f \in L^1([0,1])$, continuity of $\dis{v(t)}{f} = \int_0^1 \exp(ix/t) f(x)\, dx$ in $t >0$ is clear; to check weak* continuity of $v$ at $t = 0$, suppose $t_n > 0$, $t_n \to 0$, and let $f$ be arbitrary from the dense subspace $C^1([0,1]) \subseteq L^1([0,1])$; integration by parts gives $\dis{v(t_n)}{f} \to 0$ and, since $\norm{v(t_n)}{L^\infty} = 1$, a standard variant of the Banach-Steinhaus theorem (\cite[Section 3.3, Proposition 2]{Zeidler:95}) yields the pointwise convergence $v(t_n) \to 0 = v(0)$ on all of $L^1([0,1])$, thus the weak* continuity of $v$ at $t = 0$.)} The author would like to use this opportunity to correct a slight slip of argument in a related discussion in \cite[the paragraph below Definition 2.1]{HO:19} for periodic solutions on the torus $\T$. The claim $\lim_{t \to 0} \norm{u(t) - u_0}{L^1(\T)} = 0$ made there is true, but the reasoning rests on the entropy condition, which provides the weak* continuity of $t \mapsto u(t)$ and of $t \mapsto u(t)^2$ (compare with the discussion leading to \eqref{normcont} below) and the Cauchy-Schwarz inequality then yields the upper bound $(\int |u(t) - u_0| dx)^2 \leq  \int |u(t) - u_0|^2 dx = \int u(t)^2 dx - 2 \int u(t) u_0 dx + \int u_0^2 dx \to 0$ as $t \to 0$ and completes the argument.
\end{remark}

A typical phenomenon with nonlinear hyperbolic conservation laws is non-uniqueness of weak solutions to the Cauchy problem, in particular, for the Burgers equation (\cite[Section 4.4]{Dafermos2016}). As the Fornberg-Whitham equation is a non-local linear perturbation of the Burgers equation by the convolution term, it seems plausible that non-uniqueness is an issue\footnote{However, the author does currently know neither a concrete example of non-uniqueness nor of a uniqueness proof for the general weak Cauchy problem.}  there as well.  In any case, guided by the success of entropy (admissibility) conditions on weak solutions for pure partial differential conservation laws, such methods have also been  employed for the Fornberg-Whitham equation. 

\subsubsection{Weak entropy solutions}

Let us write Equation \eqref{eqn:F-Wrew2} as to resemble a scalar balance law, but with a non-local right-hand side as a ``source'',
$$
     \d_t u + \d_x\left(\frac{u^2}{2}\right) = - K' * u .
$$
Introducing an entropy-entropy flux pair $\eta, Q\col \R \to \R$, where $\eta$ is convex and $Q'(z) = \eta'(z) z$ (\cite[Section 3.2]{Dafermos2016}), we obtain for any classical solution,
$\d_t \eta(u) + \d_x Q(u) = \eta'(u) \d_t u + Q'(u) \d_x u = \eta'(u) ( \d_t u + u \d_x u) = - \eta'(u) (K' * u)$, 
thus 
\beq\label{classent}
   \d_t \eta(u) + \d_x Q(u) +  \eta'(u) (K' * u) = 0. 
\eeq
We cannot expect this equation to extend to weak solution as well, but we may note that for any bounded measurable function $u$, the various compositions with $u$ appearing in this equation are defined as locally bounded (Lebesgue) measurable functions: Indeed,   convexity of $\eta$ implies that $\eta$ is (locally Lipschitz) continuous and $\eta'$ is increasing, hence both $\eta$ and $\eta'$ are Borel measurable and locally bounded; furthermore, $z \mapsto Q'(z) = \eta'(z) z$ is measurable and locally bounded, hence also $Q$ is locally Lipschitz continuous; therefore, in all the compositions $\eta \circ u$, $\eta' \circ u$, and $Q \circ u$, the left member is Borel measurable, hence the composition is Lebesgue measurable.

Recall the following notion of \emph{admissibility} for a weak solution to the hyperbolic partial differential conservation law $\d_t u + \d_x f(u) + g(u) = 0$ with entropy-entropy flux pair $\eta$ and $Q$, $Q'(z) = \eta'(z) f'(z)$:  One replaces the differential equation $\d_t \eta(u) + \d_x Q(u) + \eta'(u) g(u) = 0$ for classical solutions, the analog of \eqref{classent}, by the inequality $\d_t \eta(u) + \d_x Q(u) + \eta'(u) g(u) \leq 0$ and requires that it holds in the distributional sense, i.e., $-(\d_t \eta(u) + \d_x Q(u) + \eta'(u) g(u))$ shall be equal to a non-negative measure; it even turns out that this inequality alone already implies that $u$ is a weak solution (see, e.g., \cite[the brief discussion following Definition 6.2.1]{Dafermos2016}). Taking this as a guideline for the Fornberg-Whitham equation and replacing \eqref{classent} accordingly, we will thus consider measurable locally bounded solutions $u$ of the distributional inequality 
\beq\label{weakent}
      \d_t \eta(u) + \d_x Q(u) + \eta'(u) (K' * u) \leq 0
\eeq
and obtain a concept that is compatible with, but in-between, those of weak and of strong solutions. Before implementing this in detail for the Cauchy problem \eqref{eqn:F-W} and \eqref{ic:F-W}, we will simplify matters by a typical reduction in the set of all possible entropies $\eta$ used in the inequality \eqref{weakent}, which is based on the observation that on finite intervals any convex function may be approximated by linear combinations of a linear functions and functions of the form $z \mapsto |z - \la|$ (cf.\ \cite[discussion of Theorem 1.5.1, page 25]{HoermanderConvexity} and \cite[Section 6.2]{Dafermos2016}). Namely, we need to consider only the so-called \emph{Kru\v{z}kov entropy-entropy flux pairs} (\cite[Equation (6.2.6)]{Dafermos2016}) with parameter $\la \in \R$ of the form
\beq\label{ourpair}
   \eta(z) = |z - \la|, \quad Q(z) = \sgn(z - \la) \frac{z^2 - \la^2}{2} = \frac{1}{2} |z - \la| (z + \la).
\eeq
We summarize the discussion so far  in the following solution concept.
\begin{definition}[Intermediate version]\label{defentropysol}
Let $u_0 \in L^\infty(\R)$. 
  A measurable function $u \col \R \times [0,\infty[ \to \R$ that is bounded on $\R \times [0,T]$ for every $T > 0$ is called a \emph{weak entropy solution} of the Cauchy problem \eqref{eqn:F-W} and \eqref{ic:F-W}, if
  \begin{align}
    0 &\leq \int_{0}^{\infty} \int_{-\infty}^{\infty} \Big( |u(x,t) - \lambda| \partial_t \phi(x,t) + \sgn( u(x,t)-\lambda) \frac{u^2(x,t) - \lambda^2}{2}\partial_x \phi(x,t)\\
    &-\sgn(u(x,t)-\lambda) \big( K'*u(\cdot,t) \big) (x) \phi(x,t) \Big) \,dx \,dt + \int_{-\infty}^{\infty} |u_0(x) - \lambda|\phi(x,0)\,dx
     \label{eqn:entropy-sol} 
  \end{align}  
 holds for arbitrary $\lambda \in \R$ and nonnegative test functions $\phi \in \D(\R^2)$.
\end{definition}

\noindent\emph{Entropy solutions are weak solutions:}
It is easy to check that the condition \eqref{eqn:entropy-sol} in Definition \ref{defentropysol}  implies \eqref{weaksol}, since  for any given $\phi$ we may choose $\lambda = -r$ and $\lambda = r$, where $r > 0$ is sufficiently large such that $|u| < r$ holds on the support of $\phi$. Thus, every weak entropy solution is a weak solution of the Cauchy problem.

\begin{remark}  It is equivalent to add  $K' * \la = K * \la' = 0$ in the convolution term of the integral \eqref{eqn:entropy-sol}, i.e., change $\sgn(u(x,t)-\lambda) \big( K'*u(\cdot,t) \big) (x) \phi(x,t)$ to $\sgn(u(x,t)-\lambda) \big( K'*(u(\cdot,t)-\lambda) (x) \big) \phi(x,t)$ there. We note this only to clarify consistency with the formulae mentioned in \cite{HO:19} and \cite{GH2018}.
\end{remark}

Observe that for any weak entropy solution $u$ (with the entropy-entropy flux pair given in \eqref{ourpair}), the term $\eta'(u) (K' * u)$ in \eqref{weakent} is a bounded measurable function, so that $ \d_t \eta(u) + \d_x Q(u)$ is equal to some signed measure. Thus, we may therefore again invoke \cite[Lemma 1.3.3]{Dafermos2016}, but now also with $\eta(u)$ in place of $u$, where $\eta$ is any convex function. In particular, we may choose $\eta$ quadratic and obtain that $t \mapsto u(t)$ and $t \mapsto u(t)^2$ both induce weak* continuous maps from $[0,\infty[$ into $L^\infty(B)$ for any relatively compact open subset $B$. We claim that 
\beq\label{normcont}
    t \mapsto u(t) \text{ is norm continuous } [0,\infty[ \to L^1(B).
\eeq
Let $t_n, t_0 \geq 0$ with $t_n \to t_0$ ($n \to \infty$). We have to show that $v_n := u(t_n)|_B$ converges to $v_0 := u(t_0)|_B$ in $L^1(B)$, which follows directly from the Cauchy-Schwarz inequality and the weak*-convergences $v_n \to v_0$ and $v_n^2 \to v_0^2$, since
$$
    \Big(\int_B |v_n - v_0| dx\Big)^2 \leq \int_B 1^2 dx \cdot \int_B |v_n - v_0|^2 dx = 
      |B| \Big( \int_B v_n^2 dx - 2 \int_B v_n v_0 dx + \int_B v_0^2 dx\Big) \to 0.
$$

The automatic continuity of weak entropy solutions with respect to time expressed in \eqref{normcont} suggests that with an initial value $u_0 \in L^1(\R) \cap L^\infty(\R)$ one might hope to obtain even $u \in C([0,\infty[,L^1(\R))$ for the weak entropy solution. Such a set-up works fine with scalar (partial differential) conservation laws (cf.\ \cite[Chapter VI]{Dafermos2016}) and turns out to be well-suited also for the Cauchy problem of the Fornberg-Whitham equation as demonstrated in \cite{GNg:16}. We therefore adapt Definition \ref{defentropysol} accordingly, in particular, the initial value may then be required to be attained directly in the form $u(0) = u_0$ and need not appear in the integral inequality.
\begin{definition}\label{defentropysol2}
Let $u_0 \in L^1(\R) \cap L^\infty(\R)$. 
  A function $u \in C([0,\infty[, L^1(\R))$ that is bounded on $\R \times [0,T]$ for every $T > 0$ is called a \emph{weak entropy solution} of the Cauchy problem \eqref{eqn:F-W} and \eqref{ic:F-W}, if $u(0) = u_0$ and
\begin{multline}
    0 \leq \int_{0}^{\infty} \int_{-\infty}^{\infty} \Big( |u(x,t) - \lambda| \partial_t \phi(x,t) + \sgn( u(x,t)-\lambda) \frac{u^2(x,t) - \lambda^2}{2}\partial_x \phi(x,t)\\
    -\sgn(u(x,t)-\lambda) \big( K'*u(\cdot,t) \big) (x) \phi(x,t) \Big) \,dx \,dt 
     \label{eqn:entropy-sol2} 
\end{multline}
 holds for arbitrary $\lambda \in \R$ and nonnegative test functions $\phi \in \D(\R \times ]0,\infty[)$.
\end{definition}

\subsubsection{Mild solutions} 

The (inviscid) Burgers equation is just \eqref{eqn:F-W} without the convolution term, in which case an alternative approach is to extend the nonlinear map $v \mapsto \d_x (v^2/2)$, $C^1_c(\R) \to L^1(\R)$, to an accretive operator in $L^1(\R)$ and to show that it generates a continuous semigroup of nonlinear contractions on $L^1(\R)$. In case of an initial value $u_0 \in L^1(\R) \cap L^\infty(\R)$ the concept for the Cauchy problem based on this approach is equivalent to that of a weak entropy solution (cf.\ \cite[Section 6.4]{Dafermos2016} or \cite[Section 5.5]{Barbu}).

Let us recall some of the basic notions from nonlinear operator theory (\cite[Chapter 3]{Barbu}) involved here, but writing it out specifically for the Banach space $L^1 := L^1(\R)$. A general (possibly multi-valued) nonlinear operator $G$ on $L^1$ is defined by a relation $G \subseteq L^1 \times L^1$. The value of $G$ at $u \in L^1$ is defined as the subset $G (u) := \{ v \in L^1 \mid (u,v) \in G\}$, the domain is $D(G) := \{ u \in L^1 \mid G(u) \neq \emptyset\}$, and the range is $R(G) := \bigcup_{u \in D(G)} G(u)$. Thus, $G = \{ (u,v) \mid u \in D(G), v \in G(u) \}$ and in the special case of only single-valued sets $G(u)$ for all $u \in D(G)$ this is the identification of a map with its graph. For $G, F \subseteq L^1 \times L^1$ and $\la \in \R$, we define $\la G := \{ (u,\la v) \mid (u,v) \in G\}$, the sum $G + F := \{(u,v+w) \mid (u,v) \in G, (u,w) \in F\}$, and the composition $G \circ F := \{ (u,w) \mid \exists v \in L^1\col (u,v) \in F \text{ and } (v,w) \in G\}$. We also set $G^{-1} := \{ (v,u) \mid (u,v) \in G\}$. 

A \emph{quasi-accretive} nonlinear operator $G$ on $L^1$ can be characterized by the property that there exists some $\om > 0$ such that  we have for $0 < \la < \frac{1}{\om}$,
$$
   \forall (u_1,v_1), (u_2,v_2) \in G\col \quad  \norm{u_1 - u_2 + \la (v_1 - v_2)}{L^1} \geq  (1 - \la \om) \norm{u_1 - u_2}{L^1},
$$
while $G$ is \emph{accretive}, if $\norm{u_1 - u_2 + \la (v_1 - v_2)}{L^1} \geq  \norm{u_1 - u_2}{L^1}$ holds for some (hence any) $\la > 0$. An accretive operator $G$ is said to be \emph{m-accretive}, if $R(I + G) = L^1$ (where $I$ denotes the identity on $L^1$); a quasi-accretive operator $G$ is \emph{quasi-m-accretive}, if $G + \om I$ is m-accretive for some $\om > 0$. 

A \emph{continuous semigroup of nonlinear operators (respectively, contractions)} on $L^1$ is a family $(S(t)_{t \geq 0}$ of maps $S(t) \col L^1 \to L^1$ such that $S(0) = I$,  $S(t_1 + t_2) = S(t_1) \circ S(t_2)$ for all $t_1, t_2 \geq 0$, the map $t \mapsto S(t)(u_0)$ is continuous $[0,\infty[ \to L^1$ for every $u_0 \in L^1$ (and, in case of contractions, $\norm{S(t)(u_0) - S(t)(v_0)}{L^1} \leq \norm{u_0 - v_0}{L^1}$ holds for all $u_0, v_0 \in L^1$ and $t \geq 0$). The semigroup is said to be generated by the quasi-m-accretive nonlinear operator $G$, if for every $u_0$ in the closure $\ovl{D(G)}$ of the domain of $G$, we have
$$
    S(t)(u_0) = \lim_{n \to \infty} \big( I + \frac{t}{n} G \big)^{-n} (u_0).
$$ 

We are now ready to formulate a solution concept for \eqref{eqn:F-W} and \eqref{ic:F-W} in terms of semigroups. Let $A$ be the generator of the solution semigroup of contractions for the (inviscid) Burgers equation and denote by $B$ the continuous linear convolution operator $L^1(\R) \to L^1(\R)$, $u \mapsto B u := K' * u$ (cf.\ Lemma \ref{prop:K-bounds}).
\begin{definition}  Suppose that $A + B$ is quasi-m-accretive and generates the continuous semigroup  $(S(t)_{t \geq 0}$  on $L^1(\R)$. If $u_0 \in L^1(\R)$, then $u(t) := S(t)(u_0)$ ($t \geq 0$) defines the \emph{mild solution} $u \in C([0,\infty[,L^1(\R))$ of the Cauchy problem \eqref{eqn:F-W} and \eqref{ic:F-W}.
\end{definition}

We recall from \cite[Section 6.4]{Dafermos2016} or \cite[Sections 3.3 and 5.5]{Barbu} that $A$ is given as the closure of the set $A_0 \subseteq L^1(\R) \times L^1(\R)$, where $A_0$ is defined to be the set of all pairs $(u,v) \in L^1(\R) \times L^1(\R)$ with $u^2/2 \in L^1(\R)$ and satisfying
$$
     \int_\R \sgn( u(x)-\lambda) \Big(  \frac{u^2(x) - \lambda^2}{2}\partial_x \vphi(x)  + v(x) \vphi(x)\Big)\, dx \geq 0
$$
for every non-negative $\vphi \in \D(\R)$ and $\la \in \R$.

\section{Strong solutions and wave breaking}

\subsection{Existence and uniqueness of strong solutions for short time}

The basic result on classical smooth solutions with initial and spatial $H^\infty$ regularity was established in \cite[Chapter 2, \pg 2]{NaumShish94}, alongside with the case of smooth periodic solutions in \cite[Chapter 3, \pg 2]{NaumShish94}. The strategy of proof there is successive approximation, starting with $u^{(0)}(x,t) := u_0(x)$ ($x, t \in \R$), in the form 
$$
   \d_t u^{(n)} + u^{(n-1)} \d_x u^{(n)} + K' * u^{(n-1)} = 0, \quad u^{(n)} |_{t=0} = u_0, 
$$
which requires in each step to solve a linear hyperbolic equation for $u^{(n)}$, given $u^{(n-1)}$. Estimates along the characteristics allow then to show convergence of the scheme as well as uniqueness and leads to the following statement, which in particular gives a classical solution.
\begin{theorem} If $u_0 \in H^\infty(\R)$, then there is some $T > 0$ such that the Cauchy problem \eqref{eqn:F-W} and \eqref{ic:F-W} possesses a unique solution $u \in C^\infty([0,T], H^\infty(\R))$.
\end{theorem}
Ten years later, the following unique existence result with initial and spatial $H^{k+1}$ regularity ($k \in \N$, thus $k+1 \geq 2$) was established in \cite[Theorem 4.1]{FS2004}, essentially by deriving a contraction argument for the map $v \mapsto u$, where $u$ solves
$$
   u_t + u u_x = - K' * v, \quad u|_{t=0} = u_0.
$$ 
\begin{theorem} Let $u_0 \in H^{k+1}(\R)$ with $k \in \N$. Then given any $T > 0$, which is smaller than some positive bound depending on $\norm{u_0}{H^{k+1}}$, the Cauchy problem \eqref{eqn:F-W} and \eqref{ic:F-W} is uniquely solvable with $u \in C([0,T],H^{k}(\R)) \cap L^\infty([0,T],H^{k+1}(\R))$. 
\end{theorem}
As far as we understand the details of the proof  in \cite{FS2004}, it is implicit in its arguments that the actual solution regularity is better than just $C([0,T],H^{k}(\R)) \cap L^\infty([0,T],H^{k+1}(\R))$, so that one obtains a strong solution. In fact, it follows from the equation that  $\d_t u = - u \d_x u - K' * u \in L^\infty([0,T], H^k(\R))$, so that $u$ is Lipschitz continuous as a map $[0,T] \to H^k(\R)$. 
 
For the periodic case, a similar result, but with spatial $H^{s+1}$  regularity for general $s \in \R$ with $s > 1/2$ and solution $u \in C([0,T], H^{s+1}(\T))$, was given in \cite[Theorem 1]{Holmes16}. In addition, continuous dependence of $u$ on the initial data $u_0 \in H^{s+1}(\T)$ is noted there explicitly. Moreover, reasoning again via the equation we have that $u \in C^1([0,T],H^s(\T))$ as well, hence $u$ is a strong solution. The method of proof in \cite{Holmes16} rests on Galerkin approximation and uses involved commutator and regularization techniques to derive the key energy estimates yielding convergence in appropriate function spaces. The well-posedness statement with spatial regularity $H^{s+1}$ ($s > 1/2$) for periodic and non-periodic cases and even for a whole class of related equations is mentioned also in  \cite[Theorem 1]{Liu2006}, but there the proof is omitted and only a vague reference to a ``standard iteration scheme combined with a closed energy estimate'' is made. 

The following result from \cite[Theorem 1.1]{HolTho17} holds for both cases, i.e., with the spatial variable in $\T$ or $\R$, and extends well-posedness to spatial regularity measured in the Besov scales $B^{s+1}_{s,r}$ in place of merely $H^{s+1} = B^{s+1}_{2,2}$.  
\begin{theorem} Let $u_0 \in B^{s+1}_{2,r}$ with $s > 1/2$, $1 < r < \infty$ or $s +1 = 3/2$, $r = 1$. Then for any $0 < T < c / \norm{u_0}{B^{s+1}_{2,r}}$, where $c$ is some positive constant depending only on $s$, the Cauchy problem \eqref{eqn:F-W} and \eqref{ic:F-W} is uniquely solvable with $u \in C([0,T],B^{s+1}_{2,r})$. Furthermore, the map $u_0 \mapsto u$ is continuous $B^{s+1}_{2,r} \to C([0,T],B^{s+1}_{2,r})$.
\end{theorem}
We obtain again also $u \in C^1([0,T],B^{s}_{2,r})$ and thus have a strong solution. The proof starts with a regularizing sequence $(u_0^{(n)})_{n \in \N}$ of the initial value $u_0$, putting $u^{(0)} := 0$, and defining $u^{(n)}$ ($n \geq 1$) successively as the solution of the linear hyperbolic Cauchy problem
$$
   \d_t u^{(n)} + u^{(n-1)} \d_x u^{(n)} = - K' * u^{(n-1)}, \quad u^{(n)}|_{t=0} = u_0^{(n)}.
$$
It is then shown that energy estimates hold for short enough time $T > 0$ and allow for extraction of a convergent subsequence which can be used to define a solution. Again commutator estimates involving the regularization are crucial in the process.

\begin{remark}
In case of $u_0 \in H^2(\R)$ one can give an alternative proof for the unique existence of a short-time solution with spatial $H^2$ regularity based on Kato's semi-group approach for semi-linear evolution equations. This fact was indicated very briefly in \cite{Holmes16,HolTho17} after the basic well-posedness statements. The key elements and a sketch of this are provided in the introductory section of \cite{Haziot2017} and was worked out in more detail in the first part of the proof of Theorem 1 in \cite{Wei18}.
\end{remark}

\subsection{Wave breaking for strong solutions} In contrast to well-posedness results, an analysis of wave breaking does not require to strive for statements with lowest possible regularity of the initial value. In a way, it is even more impressive to see smooth initial wave profiles leading eventually to wave breaking.

The first clear indication that wave breaking may indeed happen for solutions of the Fornberg-Whitham equation was given already in \cite{Seliger68}, where a sketch of arguments was provided including a quantitative asymmetry condition in terms of the minimum and maximum slopes occurring in the initial wave profile (see also \cite[Section 13.14]{WhithamBook}). The arguments for a wave breaking result given later in \cite{NaumShish94} picked up the basic strategy from \cite{Seliger68}, namely to look at the time development of the locations with minimum and maximum slope in a solution and to consider these as curves in the spatial domain. However, the reasoning in  \cite{NaumShish94} is not mathematically complete, as explained in \cite{ConstantinEscher1998}, where the first rigorous proof of a wave breaking result was achieved. A main issue was that one cannot guarantee a time-dependent choice of the minimal or maximal slope location that is smooth with respect to time. The key to overcome this obstacle is a theorem on the evolution of extrema proved in  \cite[Theorem 2.1]{ConstantinEscher1998} (see also \cite[Appendix 6.3.2]{Constantin2011} or \cite[Page 104, Theorem 5]{ConstEschJohnVill2016}), which has by now become a standard tool in the analysis of wave breaking and that we state therefore here as a lemma. 
\begin{lemma}\label{extrevol} Let $T > 0$ and $v \in C^1([0,T[,H^2(\R))$. Then for every $t \in [0,T[$ there is some $\xi(t) \in \R$ such that
$$
       m(t) := \inf_{x \in \R} \d_x v(x,t) = \d_x v(\xi(t),t).
$$
The function $t \mapsto m(t)$ is locally Lipschitz continuous, thus differentiable almost everywhere, and satisfies 
$$
    m'(t) = \d_t \d_x v(\xi(t),t) \quad \text{for almost every } t \in \, ]0,T[.
$$
\end{lemma}
The analogous statement is true with the supremum in place of the infimum. 
A formulation for the periodic case is slightly simpler due to compactness of the torus (\cite[Lemma 3.1]{GH:18}).

The first part in the definition of wave breaking at time $T$ according to \eqref{wavebreaking} requires that  $\norm{u(t)}{L^\infty}$ stays bounded as $t \to T$. A nice proof of this fact for solutions with spatial $H^2$ regularity is given in \cite[Proposition 2]{Haziot2017} based on an adaptation of the above lemma for the extrema of $v$ rather than of $\d_x v$. We recall the statement.
\begin{proposition}\label{propsupbound} If $u_0 \in H^2(\R)$ and $T > 0$ is the maximal life span of the corresponding unique solution $u$, then we have
$$
     \sup_{t \in [0,T[} \norm{u(t)}{L^\infty} < \infty.
$$
\end{proposition}

To prove that wave breaking actually occurs one has to show that there is a certain class of initial values $u_0 \in H^2(\R)$ such that $\norm{\d_x u(t)}{L^\infty}$ inevitably blows up as $t$ approaches the maximal life span $T$. We sketch out a basic strategy for such a proof attempt employing Lemma \ref{extrevol}:

\emph{Step 1:} Suppose $u_0 \in H^3(\R)$ and $T > 0$ is the maximal life span of the corresponding unique solution $u \in C([0,T[,H^3(\R)) \cap C^1([0,T[,H^2(\R))$. {\small (Note that we had to assume $H^3$ regularity in order to meet the regularity requirement $C^1([0,T[,H^2)$ for the function $v$ as in  Lemma \ref{extrevol}.)} For every $t \in [0,T[$ we define
$$
  m_1(t) := \inf_{x \in \R} \d_x u(x,t), \quad m_2(t) := \sup_{x \in \R} \d_x u(x,t)
$$
and $\xi_1(t), \xi_2(t) \in \R$ such that
$$
   m_1(t) = \d_x u(\xi_1(t),t), \quad m_2(t) = \d_x u(\xi_2(t),t)
$$
holds. The regularity of $u$ allows us to differentiate Equation \eqref{eqn:F-W} with respect to $x$, which yields
$$
   u_{t x} + u_x^2 + u u_{x x} + K * u_{x x} = 0.
$$
Upon observing that $u_{x x}(\xi_j(t),t) = 0$ holds by definition of $\xi_j(t)$, we evaluate this equation at $(\xi_j(t),t)$ and obtain
\beq\label{mjequ}
   m_j'(t) + m_j(t)^2 + (K * u_{x x}(t))(\xi_j(t)) = 0 \quad \text{for almost all } t \in [0,T[.
\eeq
The convolution term can be estimated from below upon an integration by parts (recalling  $K'(y) = - \sgn(y) e^{- |y|}/2$) in the following way
\begin{multline*}
   (K * u_{x x}(t))(\xi_j(t)) = - \int_{-\infty}^\infty K'(y) u_x(\xi_j(t) - y, t)\, dy\\ 
   = - \frac{1}{2} \int_{-\infty}^0 e^y u_x(\xi_j(t) - y, t)\, dy + \frac{1}{2} \int_0^\infty e^{-y} u_x(\xi_j(t) - y, t)\, dy
   \geq - \frac{m_2(t)}{2} \int_{-\infty}^0 e^y\, dy + \frac{m_1(t)}{2} \int_0^\infty e^{-y} \, dy\\
   = \frac{1}{2} (m_1(t) - m_2(t)).
\end{multline*}
Inserting this into \eqref{mjequ} gives the two differential inequalities
\beq\label{mjinequ}
   m_j'(t) \leq - m_j(t)^2 + \frac{1}{2} (m_2(t) - m_1(t)) \quad \text{for almost all } t \in\, ]0,T[ \text{ and } j =1,2.
\eeq

\emph{Step 2:} Suppose that 
\beq\label{m120}
m_1(0) + m_2(0) + S \leq 0 
\eeq
holds for some $S \geq 1$. Adding the two inequalities in \eqref{mjinequ} and observing $m_1 \leq m_2$ then yields
$$
  (m_1 + m_2)' \leq -m_1^2 - m_2^2 + m_2 - m_1 = (m_2 - m_1)(1 + m_1 + m_2) - 2 m_2^2
  \leq - m_2^2,
$$
which therefore in combination with \eqref{m120} gives
$$
  \forall t \in [0,T[\col \quad m_1(t) + m_2(t) + S \leq 0.
$$
We use this now in the inequality \eqref{mjinequ} for $j=1$ and obtain
$$
   m_1' \leq - m_1^2 + \frac{m_2}{2} - \frac{m_1}{2} \leq - m_1^2 + \frac{-S - m_1}{2} - \frac{m_1}{2}  
   = - \left(m_1 + \frac{1}{2}\right)^2 + \frac{1}{4} - \frac{S}{2} \leq - \left(m_1 + \frac{1}{2}\right)^2,
$$
which also implies
$$
    \left( m_1 + \frac{1}{2}\right)' \leq - \left( m_1 + \frac{1}{2}\right)^2.
$$

\emph{Step 3:} Putting $M(t) := m_1(t) + \frac{1}{2}$ we have $M(0) = m_1(0) + \frac{1}{2} \leq - S - m_2(0) + \frac{1}{2} < 0$ {\small (since $m_2(0) \geq 0$; otherwise, we could not have $u_0 \in L^2(\R)$)} and $M'(t) \leq - M(t)^2$, which means
$$
    \diff{t} \left( \frac{1}{M(t)} \right) = - \frac{M'(t)}{M(t)^2} \geq 1, \quad M(0) < 0,
$$
and thus implies
$$
    0 \geq \frac{1}{M(t)} \geq \frac{1}{M(0)} + t \quad (0 \leq t < 1/|M(0)| =: t_* \leq T).
$$
We conclude that $M(t) \to - \infty$ as $0 < t \to t_*$, hence $t_* \geq T$, thus $t_* = T$, and $\norm{\d_x u(t)}{L^\infty}$ cannot stay bounded as $t$ approaches $T$. 

We may thus state the following wave breaking result corresponding to \cite[Theorem 3.2]{ConstantinEscher1998} with two slight differences: First, availability of more general well-posedness results allows for less regular initial data; second, we discussed here only the specific convolution kernel $K(x) = \exp(-|x|)/2$ and not the whole class of nonzero symmetric kernel functions $K \in C(\R) \cap L^1(\R)$ that are decreasing on $[0,\infty[$. 
\begin{theorem}  If $u_0 \in H^3(\R)$ satisfies
$$
     \inf_{x \in \R} u_0'(x) + \sup_{x \in \R} u_0'(x) \leq -1,
$$
then we observe wave breaking for the unique solution of the Cauchy problem \eqref{eqn:F-W} and \eqref{ic:F-W} with initial value $u_0$. 
\end{theorem}

Note that the divergence in Step 3 of the above chain of reasoning ultimately rests on the extra condition \eqref{m120} and this is the prototype of an \emph{initial wave profile asymmetry} mentioned in the introduction to the current subsection. Unfortunately, a direct comparison of quantitative wave breaking conditions used in various results on wave breaking in the literature is somewhat impaired by the fact that these certainly have to depend on the exact conventions used for scaling and signs in the Fornberg-Whitham equation. 

The reasoning in the wave breaking result of \cite[Section 3]{Haziot2017} is similar to the above, but uses refined estimates in Steps 1 and 2  and gives a  sufficient condition on the minimal and maximal slopes of $u_0$ weighted by a real parameter from a bounded interval. For the periodic case, a wave breaking result along the lines of the above theorem is proved in \cite[Section 3]{GH:18}. All the previous sufficient conditions on $u_0$ have been shown in \cite[Theorem 2]{Wei18} to be special cases of one more general condition that still leads to wave breaking. The proof departs from the above strategy after inequalities \eqref{mjinequ} at the end of Step 1 and succeeds to produce subtle bounds on an appropriate linear combination of $m_1^2$ and $m_2 - m_1$, which leads to a sufficient condition of the structure 
$$
   m_1(0) < \min (- c_1, - c_2(1 + \sqrt{1 + c_3 m_2(0)}))
$$
with positive constants $c_j$ depending on the precise scaling and sign conventions used in the Fornberg-Whitham equation.

For smooth periodic solutions the blow-up of $u_x$ in finite time is shown in \cite[Theorem 2]{Liu2006}, if  $- \inf_{x \in \R} u_0'(x)$ is sufficiently large\footnote{This is a bit reminiscent of the classical blow-up condition for initial values with the Burgers equation, where negative slope causes characteristics to intersect.}, and a similar result is shown in \cite[Section 4]{GNg:16} for the non-periodic case with initial value $u_0 \in C^1(\R) \cap L^1(\R)$. Both of these proofs use elaborate estimates along characteristics and the sufficient conditions require in particular domination of $\norm{u_0}{L^\infty}$ or $\norm{u_0}{L^1}$, respectively. To justify these results strictly as proofs of wave breaking, one should also guarantee boundedness of $\norm{u(t)}{L^\infty}$ as $t$ approaches the critical blow-up time $t_*$. This is implicitly so in \cite{GNg:16}, since there $u$ is supposed to be the unique weak entropy solution with initial value $u_0$.

Some numerical case studies of wave breaking as the formation of shocks in weak solutions on the torus are contained in \cite{HO:19}.  They suggest that only negative infinities of $u_x$ are developing and that $u_x$ stays bounded from above at the moment of wave breaking. This also finds support by the Oleinik type inequality proved in \cite[Lemma 2.1]{GNg:16} (see also \eqref{Oleinik} below) for weak entropy solutions on the real line and can be shown directly for strong solutions with spatial $H^3$ regularity by calling on Lemma \ref{extrevol}. We will discuss this below after first listing a few basic results about convolution with $K'$ that will also be useful for the application of semigroup theory later on.
\begin{lemma}\label{prop:K-bounds}
  The linear operator $u \mapsto K'*u$ 
  \begin{itemize}
    \item is bounded from $L^q(\R)$ to $L^p(\R)$ for all $p, q \in \R$ with $1 \leq q \leq p \leq \infty$,
    \item maps $BV(\R)$ into $W^{1,\infty}(\R) \cap W^{1,1}(\R)$,
    \item and for any $u \in L^\infty(\R)$ one has that
\begin{equation}\label{derivativeconvolutionbound}
  \sup |\partial_x (K'*u)| \leq   2 \| u \|_\infty.   
  \end{equation}
  \end{itemize}
\end{lemma}

\begin{proof}
  For the first point, as $K' \in L^r(\R)$ for every $1 \leq r \leq \infty$, we obtain $\| K'*u \|_p \leq \| K' \|_r \| u \|_q$, if $1 \leq q := r p / (r + r p - p) \leq p$, from Young's convolution inequality (\cite[Proposition 8.7]{Folland1999}).

  For the second point, we note that $\partial_x (K' * u) = K' * Du$, where we may interpret $Du$ as the BV derivative of $u$, which is a finite measure by assumption (cf.\ \cite[Definition 7.1]{Leoni-2017}). We can then apply the version of Young's inequality for convolution with measures (cf. \cite[Proposition 8.49]{Folland1999}) to obtain (with $1 \leq p \leq \infty$ arbitrary)
$$
   \| K'*Du\|_p \leq |Du|(\R) \cdot \|K'\|_p,
$$
where $|Du|$ denotes the total variation measure associated with $Du$.
  
For the last point, note that $K'' = K - \delta$ and we therefore obtain
$$
   \sup |\partial_x (K' * u)| = \| K''  * u \|_\infty  \leq 
     \| K * u\|_\infty + \| u \|_\infty \leq (\| K \|_1 + 1)  \| u \|_\infty = 2  \| u \|_\infty.
$$
\end{proof}

Assuming initial data $u_0 \in H^3(\R) \subset W^{1,1}(\R) \subset BV(\R) \subset L^1(\R) \cap L^\infty(\R)$ there is some maximal life span $T > 0$ of a unique strong solution $u \in C([0,T[,H^3(\R)) \cap C^1([0,T[,H^2(\R))$. In case $T < \infty$, $u$ ceases to be a strong solution in the form of wave breaking at time $t = T$, i.e., $\sup_{0 \leq t < T}\| u(t)\|_\infty$ is bounded while $\limsup_{t \uparrow T}\| u_x(t)\|_\infty = \infty$, in fact, $\inf_{x \in \R} \d_x u (x,t) \to - \infty$ as $t \to T$ (cf.\ \cite[Proposition 1]{Haziot2017}). Furthermore, there are sufficient conditions on the initial wave profile $u_0$ to definitely cause $T < \infty$, thus wave breaking occurs even for smooth initial values. 
Combined with the following proposition we may deduce that in case of wave breaking
$$
  \text{a shock in the spatial wave profile can only form as a downward jump} 
$$
(in the direction of growing $x$). 
\begin{proposition} If $u_0 \in H^3(\R)$ and $T$ is the maximal life span of the corresponding unique strong solution $u \in C([0,T[, H^3(\R)) \cap C^1([0,T[,H^2(\R))$ to \eqref{eqn:F-W} and \eqref{ic:F-W},  then 
\begin{equation}
   \sup_{0 \leq t < T} \sup_{x \in \R} \d_x u (x,t) < \infty.
\end{equation}
\end{proposition}

\begin{proof} We put $M(t) := \sup_{x \in \R} u_x(x,t)$ and may call on Lemma \ref{extrevol} to deduce the following three facts:  $M$ is  differentiable almost everywhere on $[0,T[$; for every $t \in [0,T[$ there exists $\xi(t) \in \R$ such that $M(t) = u_x(\xi(t),t)$; and we have the relation
$$
      M'(t) = u_{tx}(\xi(t),t) \quad \text{a.e.\ on } [0,T[.
$$ 
Noting that $u_{xx}(\xi(t),t) = 0$ we obtain upon differentiation in \eqref{eqn:F-W} from \eqref{derivativeconvolutionbound} in Lemma \ref{prop:K-bounds} 
$$ 
  M'(t) = - M(t)^2 - 0 - \d_x \big(K \ast u_x(.,t)\big)(\xi(t)) \leq 
     -M(t)^2 + 2  \|u(t)\|_\infty \quad \text{for almost every $t \in [0,T[$}.
$$ 
By Proposition \ref{propsupbound} we have $c^2 := 2 \sup_{0 \leq t < T} \|u(t)\|_\infty < \infty$ (with $c \geq 0$) and obtain
$$
    M'(t) \leq c^2 - M(t)^2.
$$
Note that $0 \leq M(0) = \sup_{x \in \R} u_0'(x) < \infty$, since $u_0 \in L^2(\R) \cap C^1(\R)$ and $u_0' \in H^2(\R) \subset L^\infty(\R)$. Now  consider the solution $y$ to the initial value problem $y(0) = M(0)$, $y'(t) = c^2 - y(t)^2$: It is constant, if $M(0) = c$; in case $M(0) <  c$, we have  $y(t) = c \tanh(c t + \alpha)$ with $\tanh(\alpha) = M(0)/c < 1$; and in case $M(0) > c$, we have  $y(t) = c \coth(c t + \alpha)$ with $\coth(\alpha) = M(0)/c > 1$. In any case, $y(t)$ exists for all $t \in [0,T]$ and is bounded. Since $M(0) = y(0)$, $M' \leq c^2 - M^2$, and $y' = c^2 - y^2$, an application of the comparison theorem for ordinary differential equations (e.g., \cite[Lemma 16.4]{Amann1990}) yields $M(t) \leq y(t)$ for $t \in [0,T[$, thus $M$ is bounded from above. 
\end{proof}

\begin{remark} The above proposition does not tell whether the height of a downward shock that was formed due to wave breaking will stay bounded or decrease as time progresses.  As the existence of bounded, piecewise smooth, traveling wave solutions with entropic jump discontinuities shows, we cannot in general expect a decrease of the shock height with time for entropy solutions in the sense of Definition \ref{defentropysol} (see the paragraph on heteroclinic connections in \cite[Section 3]{FS2004}  or  \cite{GH2018}). 
\end{remark}

\section{Weak solutions from entropy concepts, semigroup methods, or traveling waves}

\subsection{Weak entropy solutions}
First indications that the \emph{method of vanishing viscosity} produces a convergent scheme seem to be given in \cite[Chapter 5, \pg 2 and \pg 3]{NaumShish94}, although their notion of \emph{generalized solution} remains vague and uniqueness is not addressed. The basic strategy of vanishing viscosity was later used to produce the following rigorous statement on weak entropy solutions for the Fornberg-Whitham equation in \cite[Theorem 4.2]{FS2004}, where spatial $BV$ regularity is assumed. {\small (The original formulation does not describe the relation of the solution $u$ with the initial value $u_0$, but we know from our previous discussion of the solution concepts that $u(0) = u_0$ holds in the sense of $u \in C([0,\infty), L^1(\R))$, since $u_0 \in BV(\R) \subseteq L^1(\R) \cap L^\infty(\R)$.)} The uniqueness follows in the proof given in \cite{FS2004} from an intermediate $L^1$-stability result (see also \eqref{L1stab} below).
\begin{theorem} If $u_0 \in BV(\R)$, then there is a unique weak entropy solution $u$ (in the sense of Definition \ref{defentropysol2}) to
the Cauchy problem \eqref{eqn:F-W} and \eqref{ic:F-W}, which in addition satisfies $u \in L^\infty_\loc([0,\infty[,BV(\R))$.
\end{theorem}
This result was extended in \cite[Theorem 1.2]{GNg:16} to the general case described by the situation in Definition \ref{defentropysol2}. We note that although \cite[Definition 1.1]{GNg:16} does not explicitly specify the precise quality assumed of the initial value $u_0$ and the formulation in \cite[Theorem 1.2]{GNg:16} speaks only of $u_0 \in L^1(\R)$, we have some doubt whether the concept and all the proof details are true without having also $u_0 \in L^\infty(\R)$ a priori. In any case, our formulation of the main result with the a priori requirement $u_0 \in L^1(\R) \cap L^\infty(\R)$ is certainly covered by and coherent with \cite{GNg:16}.
\begin{theorem}\label{GNThm} Given $u_0 \in L^1(\R) \cap L^\infty(\R)$, the Cauchy problem \eqref{eqn:F-W} and \eqref{ic:F-W} has a unique weak entropy solution $u$ (in the sense of Definition \ref{defentropysol2}). It satisfies the Oleinik type inequality
\beq\label{Oleinik}
    \forall t > 0, \forall x, y \in \R, x < y \col \quad u(y,t) - u(x,t) \leq \left( \frac{1}{t} + 2 + 2t (1 + 2 e^t \norm{u_0}{L^1}) \right) (y - x).
\eeq
Moreover, the following $L^1$-stability holds: If $v$ is the weak entropy solution corresponding to the initial value $v_0 \in L^1(\R) \cap L^\infty(\R)$, then
\beq\label{L1stab}
   \forall t > 0\col \quad \norm{u(t) - v(t)}{L^1} \leq e^t \norm{u_0 - v_0}{L^1}.
\eeq
\end{theorem}
Of course, as a corollary of \eqref{L1stab} (with $v_0 = 0$) we obtain the following estimate for every $t \geq 0$:
$$
    \norm{u(t)}{L^1} \leq e^t \norm{u_0}{L^1}.
$$

The proof of Theorem \ref{GNThm} is by a so-called \emph{flux-splitting method} and uses approximate solutions based on discretized time steps and the solution semigroup for the Burgers equation applied in each interval between these time steps. Convergence is shown by fine techniques involving estimates for the Burgers semigroup and regularity properties of solutions to the Poisson equation. Continuity of $u$ as a function into $L^1$ follows from a tightness condition of the approximate solution sequence, which is shown via energy estimates establishing H\"older regularity of the characteristics along the way. Boundedness of weak entropy solutions combined with the $L^1$ continuity of $u$ with respect to time produces an integral inequality, which implies $L^1$-stability and hence also uniqueness.

The recent publication \cite{LL2020} states results partially parallel to Theorem \ref{GNThm} and independently sketches arguments based on vanishing viscosity solutions and compensated compactness. The solutions obtained are in coherence with the current setting, although the solution concepts given in \cite[Definitions 2.6 and 2.7]{LL2020} fail to clarify details about the initial data and neither the definition of weak entropy solutions nor the main existence theorem \cite[Theorem 3.6]{LL2020} include continuity aspects of the solution with respect to time.

Well-posedness and $L^1$-stability for \emph{periodic weak entropy solutions to the Fornberg-Whitham equation} has been shown independently in \cite[Section 2]{HO:19} along the lines of Kru\v{z}kov's original paper \cite{Kruzkov:70} and with an adaptation of an older technique by Fujita and Kato for the Navier-Stokes equation based on the analytic semigroup generated by $- \eps \d_x^2$ on $L^2(\T)$.

\begin{remark} We note that \cite[Theorem 1]{GLC2018} contains an $L^1$-stability statement analogous to \eqref{L1stab}, although for strong solutions and only for times of their common existence.  There are also the following bounds for the spatial $L^\infty$ norm of a strong solution $u$ with existence time $T > 0$ and initial value $u_0 \in H^s(\R)$ ($s > 3/2$), given in \cite[Lemma 4]{GLC2018}, 
$$
   \| (K * u_x)(t)\|_{L^\infty} \leq \| u_0\|_{L^2}  \quad\text{and}\quad \|u(t)\|_{L^\infty} \leq \|u_0\|_{L^\infty} + t \|u_0\|_{L^2}.
$$
\end{remark}

\subsection{Mild solutions} Our goal here is to establish mild solutions and also the well-posedness of weak entropy solutions via the generation of a non-linear semigroup. The basic properties of the non-local term according to Lemma \ref{prop:K-bounds} allow us to see the following theorem almost as a direct application of the theory of semigroups on Banach spaces generated by non-linear operators (as described, e.g., in \cite{Barbu}).

\begin{theorem}\label{semisol} If $u_0 \in L^1(\R)$ then there exists a unique global mild solution $u \in C([0,\infty[, L^1(\R))$ of the Cauchy problem \eqref{eqn:F-W} and \eqref{ic:F-W} in the sense of semigroups. 
\end{theorem}

\begin{proof}
  As the non-local term is bounded in $L^1$, and as the inviscid Burgers term generates a non-linear contraction semigroup in $L^1$, we may consider the former a perturbation of the latter and apply the theory of semigroups generated from nonlinear operators. More precisely, let $A$ be the non-linear operator associated with the Burgers equation as in \cite[Section 3.3]{Barbu} or in \cite[Section 6.4]{Dafermos2016}, and put $Bu  := K' * u$ with domain $D(B) = L^1(\R)$. By Proposition \ref{prop:K-bounds} we have the finite operator norm $b:= \|B\| < \infty$ for $B$ as linear map  $L^1(\R) \to L^1(\R)$. We claim that $A+B$ is quasi-m-accretive on $L^1(\R)$ (in the sense of \cite[Section 3.1]{Barbu}).

  To establish this, we first note that from the accretiveness of A and boundedness of B, we have for $v_1, v_2 \in L^1(\R)$ and $0 < \lambda < 1/b$,
\begin{multline*}
    \|v_1 - v_2 + \lambda ( A(v_1) + B v_1 - A(v_2) - B v_2)\|_{1} \\  \geq
     \|v_1 - v_2 + \lambda ( A(v_1) - A(v_2))\|_{1} - \lambda \| B (v_1 - v_2) \|_{1}
    \geq (1 - \lambda b) \|v_1 - v_2\|_{1},
\end{multline*}
hence A + B is quasi-accretive.

Second, appealing to \cite[Proposition 3.3]{Barbu}, the accretive operator $A+B$ is quasi-m-accretive, if we can show surjectivity of $I + \lambda (A + B)$ for small $\lambda > 0$, e.g.\ by proving solvability of the following equation in $L^1(\R)$ for $u$ given $v$:
\[
  (I + \lambda A)^{-1} (v) - (I + \lambda A)^{-1} (\lambda B u) =   u.
\]
As long as $\lambda < 1/b$, the left-hand side is a contraction, since by accretiveness of $A$,
$$
   \| (I + \lambda A)^{-1} (\lambda B u_1) - (I + \lambda A)^{-1} (\lambda B u_2) \|_{1}  \leq
    \| \lambda B u_1 -\lambda B u_2 \|_{1} \leq \lambda b \| u_1 - u_2 \|_{1}
$$
and hence the equation is solvable. 

For the quasi-m-accretive operator $A+B$ we have by \cite[Proposition 3.6]{Barbu} that its domain $D(A+B)$ is dense in $L^1(\R)$. Therefore, the existence and uniqueness of a mild solution in the sense of non-linear semigroups with initial value $u_0 \in \overline{D(A+B)} = L^1(\R)$ now follows from \cite[Corollary 4.1]{Barbu}. 
\end{proof}

\begin{remark} 
In course of the above proof we preferred to show directly that of $A+B$ is quasi-m-accretive, while alternatively, one could also just observe quasi-accretiveness of $B$ and apply an appropriate variant of the basic perturbation result proved in \cite[Theorem 3.2]{Barbu1976} (and mentioned also in \cite[Theorem 3.1]{Barbu}).
\end{remark}

The following result gives an independent proof of the well-posedness part from Theorem \ref{GNThm}.
\begin{theorem}\label{entropysol} If $u_0 \in L^1(\R) \cap L^\infty(\R)$, then the global mild solution $u \in C([0,\infty[, L^1(\R))$ of the Cauchy problem \eqref{eqn:F-W} and \eqref{ic:F-W} according to Theorem \ref{semisol} is also a weak entropy solution.
\end{theorem}

\begin{proof} Assuming now that $u_0 \in L^1(\R) \cap L^\infty(\R)$, the equivalence of the semigroup solution according to Theorem \ref{semisol} and the entropy solution follows similarly as in the proof of \cite[Theorem 5.6]{Barbu}. We indicate a few adaptations implementing the proof variant in our case: First, since we have shown that the range of $I + \lambda (A + B)$ for small $\lambda > 0$ is all of $L^1(\R)$,  we may note that following  \cite[Theorem 4.3, Equation (4.17)]{Barbu}, the mild solution can be constructed as the limit of resolvents
\begin{displaymath}
  u(t) = \lim_{n \to \infty} \big(I + \frac{t}{n}(A+B)\big)^{-n} u_0
\end{displaymath}
uniformly in $t$ on compact intervals. Second, the resolvent-like bounds established above hold (with an appropriately changed constant $b > 0$) with respect to any $L^p$-norm, $1 \leq p \leq \infty$, since this is true for the unperturbed operator $A$, the convolution operator $B$ is bounded on $L^p(\R)$ as well (Lemma \ref{prop:K-bounds}), and the above estimates for $A+B$ were generic, i.e., without using special properties of the $L^1$-norm. In combination of these facts, it follows  that the solutions $u_\varepsilon$ to the $\varepsilon$-regularized difference equation as in \cite[Equation (5.125)]{Barbu}, namely $(u_\eps(t) - u_\eps(t - \eps))/\eps + A(u_\eps(t)) + B u_\eps(t) = 0$ for $t > \eps$ and $u_\eps(t) = u_0$ for $t < 0$, satisfy  
\begin{displaymath}
  \|u_\varepsilon(t)\|_p \leq e^{bt}\|u_0\|_p,
\end{displaymath}
uniformly in $\varepsilon > 0$, and $u_\eps(t) \to u(t)$ as $\eps \to 0$, uniformly for $t$ in a compact time interval. This uniform upper bound for $u_\eps$ allows us to enter the proof of \cite[Theorem 5.6]{Barbu} at (5.128) and follow the line of arguments there up to the end with $A+B$ always replacing $A$, which concludes the proof of our theorem.
\end{proof}

\subsection{Continuous weak traveling wave solutions}

In the theory of (local) scalar conservation laws it can be shown that continuous weak solutions are always entropy solutions (\cite[Theorem 11.13.1]{Dafermos2016}). The proof employs fine-tuned techniques from the theory of generalized characteristics and might to be out of reach in our case of a nonlocal conservation law with the Fornberg-Whitham equation. However,  for the special situation of traveling waves we show a related result below. Its hypotheis includes the case of the famous peakon solution with initial wave profile $4 \exp(-|y|/2)/3$, which gives a weak solution to the Fornberg-Whitham equation (cf.\ \cite{FB78,CLH2012} or \cite[Example 1.4]{GH2018}). In general for a traveling wave $u(x,t) = v(x - ct)$ we do not want to require $u_0 = v \in L^1(\R)$, since this would exclude many interesting cases. Thus, in the following statement we do not assume that $x \mapsto u(x,t)$ is integrable for every $t$ and resort to Definition \ref{defentropysol} instead of \ref{defentropysol2}.

\begin{proposition}\label{travprop} Any weak traveling wave solution $(x,t) \mapsto u(x,t) = v(x - ct)$ with bounded and absolutely continuous\footnote{In the sense that $v$ is differentiable almost everywhere with locally integrable derivative.} wave profile $v$ is an entropy solution in the sense of Definition \ref{defentropysol}. 
\end{proposition}

\begin{proof} From the assumption that $u(x,t) = v(x - ct)$ defines a weak solution it is not difficult\footnote{For example, along the lines of the reasoning in \cite[Subsection 2.1]{GH2018}, but here with the simplification of continuity of the wave profile at $\xi = 0$.} to derive the following equation, which holds in the sense of distributions as well as pointwise almost everywhere on $\R$:
\begin{equation}\label{distrv}
   \left( \frac{(v - c)^2}{2} +  K * v \right)' = 0.
\end{equation} 
We will show that for any $\lambda \in \R$ and nonnegative test function $\phi$ in $C^\infty_\text{c}(\R^2)$,
 \begin{multline*} 
     \int\limits_{0}^{\infty} \int\limits_{-\infty}^{\infty} |v(x - ct) - \lambda| \partial_t \phi(x,t) \,dx dt + 
     \int\limits_{0}^{\infty} \int\limits_{-\infty}^{\infty} 
           \sgn( v(x - ct)-\lambda) \frac{v^2(x - ct) - \lambda^2}{2}\partial_x \phi(x,t) \,dx dt\\
    - \int\limits_{0}^{\infty} \int\limits_{-\infty}^{\infty} 
       \sgn(v(x - ct)-\lambda)K'*(v(\cdot - ct)) (x) \phi(x,t) \,dx dt + 
    \int\limits_{-\infty}^{\infty} |v(x) - \lambda|\phi(x,0)\,dx = 0.
  \end{multline*}  
Let us denote the four integral terms on the left-hand side by $I_1, I_2, I_3, I_4$, respectively, i.e., we claim that $I_1 + I_2 - I_3 + I_4 = 0$. 

Fubini's theorem and integrating by parts with respect to $t$, gives
 \begin{multline*} 
     I_1 = \int\limits_{-\infty}^{\infty}  \int\limits_{0}^{\infty}  
     |v(x - ct) - \lambda| \partial_t \phi(x,t) \,dt dx\\
     = \int\limits_{-\infty}^{\infty} \Big(  \int\limits_{0}^{\infty} 
      \Big(\sgn(v(x - ct) - \lambda) c v'(x - ct) \phi(x,t) \Big) \,dt +  
       |v(x - ct) - \lambda| \phi(x,t) |_{t=0}^{t = \infty}  \Big)dx\\
    = \int\limits_{-\infty}^{\infty}  \int\limits_{0}^{\infty}  \Big(\sgn(v(x - ct) - \lambda) 
           c v'(x - ct) \phi(x,t) \Big) \,dt dx - \int\limits_{-\infty}^{\infty}  |v(x) - \lambda| \phi(x,0)  \,dx,
  \end{multline*}
where we already observe that the last term cancels $I_4$.
  
In $I_2$  we observe that $f(y) = \sgn(y - \lambda) (y^2 - \lambda^2)/2$ is differentiable with derivative $f'(y) = \sgn(y - \lambda) y$ in an integration by parts to obtain
\begin{multline*}
  I_2 = \int\limits_{0}^{\infty} \int\limits_{-\infty}^{\infty} 
           \sgn( v(x - ct)-\lambda) \frac{v^2(x - ct) - \lambda^2}{2}\partial_x \phi(x,t) \,dx dt\\
    = -  \int\limits_{0}^{\infty} \int\limits_{-\infty}^{\infty} 
           \sgn( v(x - ct)-\lambda) v(x - ct) v'(x - ct) \phi(x,t) \,dx dt.
\end{multline*}
Summing up, we find
\begin{multline*}
   I_1 + I_2 - I_3 + I_4\\  = 
    \int\limits_{0}^{\infty} \int\limits_{-\infty}^{\infty} 
           \sgn( v(x - ct)-\lambda) \underbrace{\Big( (c - v(x - ct)) v'(x - ct) - (K' * v) (x - ct) \Big)}_{= - ((v-c)^2/2)' - K' * v = 0 \text{ a.e.}} \phi(x,t) \,dx dt = 0.
\end{multline*}
\end{proof}

The hypothesis of absolute continuity in the previous proposition certainly would allow for non-smoothness in $v$ harsher  than the Lipschitz continuous corner singularity in the example of the peakon solution.  An interesting question is whether absolutely continuous functions $v$ with a cusp at some location $x_0 \in \R$, where the derivative is locally integrable but unbounded, qualify as initial values of weak traveling solutions $u$. If we have, in addition, $v \in L^1(\R) \cap L^\infty(\R)$ this can be ruled out immediately: The proof of Proposition \ref{travprop} shows that $u$ would be also a weak entropy solution in the sense of Definition \ref{defentropysol2} and we could construct a contradiction to the Oleinik type estimate \eqref{Oleinik} in Theorem \ref{GNThm} for any $t > 0$ near the translated cusp location $x_0 + c t$.

\begin{example}  In  \cite{CLH2012} the authors construct an interesting class of examples of bounded continuous traveling waves with a cusp and satisfying Equation \eqref{eqn:F-W} in the pointwise classical sense everywhere on $\R^2$ except for the straight line $x = ct$. 
We consider the particular case with parameters $A=0$ and $c > 4/3$ in \cite[Theorem 2.4(i), Theorem 2.5(iii), and Case III in Section 3]{CLH2012} and obtain the traveling wave $u(x,t) := v(x - ct)$, where $v$ is a bounded continuous function on $\R$ that is $C^3$ off $0$ and satisfies
\begin{equation}\label{travwaveode}
   \left(1 - \frac{d^2}{d\xi^2}\right) \left( \frac{(v - c)^2}{2}\right)' +  v' = 0 \quad \text{ on } \R \setminus\{ 0\}.
\end{equation}
Furthermore,  $0 < v \leq c$, $v(0) = c$, $\xi \mapsto v(\xi)$ is strictly increasing for $\xi < 0$, $v(-\xi) = v(\xi)$, $\lim_{\xi \to \pm \infty} v(\xi) = 0$, and we have, with the constant $b := 4 |c|^{3/2} \sqrt{c - 4/3}> 0$,
\begin{equation}\label{asymptv}
   v(\xi) = c - 2 b |\xi|^{1/2} + O(|\xi|) \quad\text{and}\quad v'(\xi) = - b \sgn(\xi) |\xi|^{-1/2} + O(1)  \quad (\xi \to 0).
\end{equation}
Note that $v$ is absolutely continuous. In fact, an inspection of the change of variables in the construction of \cite[Section 3, Case III]{CLH2012} shows that we have even $v \in W^{1,1}(\R) \subset BV(\R) \subset L^1(\R) \cap L^\infty(\R)$. Therefore, the argument presented above already shows that $u$ cannot be a weak solution.

However, let us add here also a more direct reasoning why Equation \eqref{travwaveode}, valid pointwise for $\xi \neq 0$,  cannot guarantee that the initial wave profile $v$ defines a global weak solution $u$ of the Fornberg-Whitham equation. Similar arguments may be applicable to other cases of parameters in this example class as well.

We will show that the precise asymptotic information about $v$ and $v'$ near $\xi = 0$ according to \eqref{asymptv} allows us to draw the following conclusion: If $v$ has all the properties specified above and the left-hand side of  \eqref{travwaveode} is the restriction of a distribution on $\R$ which vanishes on $\R\setminus\{0\}$, then
\begin{equation}\label{travwavedistr}
  \left( \frac{(v - c)^2}{2} +  K * v \right)' = - 4 b^2 K'.
\end{equation}
Since Equation \eqref{travwavedistr} is in contradiction to \eqref{distrv}, we may then conclude that $u(x,t) = v (x - ct)$ cannot define a weak solution of the Fornberg-Whitham equation. 
 
To prove \eqref{travwavedistr}, we first note that due to \eqref{travwaveode} the distribution $(1 - \frac{d^2}{d\xi^2}) ((v - c)^2/2)' + v'$ has support in the singleton set $\{ 0\}$, thus equals a finite linear combination of derivatives of the Dirac distribution $\delta$ (concentrated at $\xi = 0$). 
Recall that $v$ is globally continuous, even $C^3$ outside $\xi = 0$, and by \eqref{asymptv} the derivative $v'$ is locally integrable; hence also $((v-c)^2/2)' = (v-c) v'$ is locally integrable. Therefore the order of the Delta derivatives can be at most $1$, i.e., there are $\lambda_0, \lambda_1 \in \R$ such  that
$$
   \left(1 - \frac{d^2}{d\xi^2}\right) \Big(\frac{(v - c)^2}{2}\Big)' + v' = \lambda_0 \delta + \lambda_1 \delta'.
$$
Upon convolution with $K$ we obtain
\begin{equation}
   ((v - c)^2/2)' + K' * v = \lambda_0 K + \lambda_1 K',
\end{equation}
which implies $((v - c)^2/2)' = \lambda_0 K + \lambda_1 K' - K' * v \in L^1(\R)$ and upon integration over $\R$ that
$$
     \lambda_0 = \int_{-\infty}^\infty (\lambda_0 K + \lambda_1 K' - K' * v) d \xi = \int_{-\infty}^\infty ((v - c)^2/2)' d\xi = 0,
$$
since $\lim_{\xi \to \pm \infty} (v(\xi) - c)^2/2 = c^2/2$. Thus, we are left with the equation
\begin{equation}\label{weaktravconv}
   ((v - c)^2/2)' + K' * v = \lambda_1 K'.
\end{equation}
Considering again \eqref{asymptv},  when $\xi \to 0$ we have
\begin{multline*}
   ((v(\xi) - c)^2/2)' =\\ (v(\xi) - c) v'(\xi) = (- 2 b |\xi|^{1/2} + O(|\xi|)) (- b \sgn(\xi) |\xi|^{-1/2} + O(1))\\ = 2 b^2 \sgn(\xi) + O(|\xi|^{1/2}),
\end{multline*} 
hence $((v-c)^2/2)'$ has a jump of height $4 b^2$ at $\xi = 0$. Recalling $K'(\xi) = - \exp(-|\xi|) \sgn(\xi)/2$, using the continuity of $K' * v$ and of $K$ when taking the differences in \eqref{weaktravconv} as $\xi \to \pm0$ we finally conclude that $4 b^2 = - \lambda_1$. 

\end{example}

\begin{remark} It can be shown (cf.\ \cite[Section 3]{FS2004}  or  \cite{GH2018}) that there are bounded, piecewise smooth, traveling waves with an entropic jump discontinuity that are weak entropy solutions in the sense of Definition \ref{defentropysol}. 
\end{remark}

\bigskip

\paragraph{\textbf{Acknowledgments:}} The basic strategy leading to the proofs given here of Theorems \ref{semisol} and \ref{entropysol} was suggested by Alberto Bressan during a visit at the University of Vienna. The author expresses his sincere thanks to him, but also to Ryan Murray for jointly setting up and discussing first details regarding these theorems as well as Lemma \ref{prop:K-bounds}.

\bigskip

\bibliography{FWBP}

\end{document}